\documentclass[a4paper,12pt]{article}

\usepackage[french, ngerman, italian, english]{babel}
\usepackage{amsmath,amsthm, amssymb}
\usepackage{setspace}
\usepackage{anysize}
\usepackage{url}
\usepackage{graphicx}
\usepackage[colorlinks=true]{hyperref}
\usepackage{mathptmx}
\usepackage{paralist}
\usepackage{verbatim}
\usepackage[utf8]{inputenc}

\theoremstyle{plain}
\newtheorem{thm}{\bf Theorem}[section]

\newtheorem{prop}[thm]{\bf Proposition}
\newtheorem{lemma}[thm]{\bf Lemma}
\newtheorem{corollary}[thm]{\bf Corollary}

\theoremstyle{definition}
\newtheorem{definition}[thm]{\bf Definition}
\theoremstyle{remark}
\newtheorem{remark}[thm]{\bf Remark}
\newtheorem{example}[thm]{\bf Example}
\theoremstyle{example}

\def \HF{{\operatorname{HF}_{\bar{A}(G)}}}
\def \HP{{\operatorname{HP}_{\bar{A}(G)^{(2)}}}}
\def \ara{{\operatorname{ara}}}
\def \rank{{\operatorname{rank}}}
\def \height{{\operatorname{ht}}}
\def \reg{{\operatorname{reg}}}
\def \pd{{\operatorname{pd}}}
\def \mm{{\mathfrak{m}}}
\def \aa{{\alpha}}
\def \bb{{\beta}}
\def \cc{{\gamma}}
\def \NN{\mathbb N}

\def \L{\mathcal L}

\def \P{\mathcal P}
\def \AG{\bar{A}(G)}
\def \gdim{\operatorname{gdim}}

\def\ini{\operatorname{\rm in}}
\def\GB{Gr\"obner basis}
\def\tv{\ensuremath{\sqcup}}
\def\tw{\ensuremath{\sqcap}}

\begin{document}
\title{Koszulness, Krull Dimension and Other Properties of Graph-Related Algebras }
\author{Alexandru Constantinescu \\
\footnotesize Dipartimento di Matematica\\
\footnotesize Univ. degli Studi di Genova, Italy\\
\footnotesize \url{constant@dima.unige.it}
\and \setcounter{footnote}{3} Matteo Varbaro \\
\footnotesize Dipartimento di Matematica\\
\footnotesize Univ. degli Studi di Genova, Italy\\
\footnotesize \url{varbaro@dima.unige.it}}
\date{{\small \today}} 
\maketitle

\begin{abstract}
\noindent 
The algebra of basic covers  of a graph $G$, denoted by $\AG$, was introduced by J\"urgen Herzog as a suitable quotient of the vertex cover algebra.
In this paper we show that if the graph is bipartite then $\AG$ is a homogeneous algebra with straightening laws and thus is Koszul.
Furthermore, we compute the Krull dimension of $\AG$ in terms of the combinatorics of $G$. As a consequence we get new upper bounds on the arithmetical rank of monomial ideals of pure codimension $2$. Finally, we characterize the Cohen-Macaulay property and the Castelnuovo-Mumford regularity of the edge ideal of a certain class of graphs.
\end{abstract}

\section*{ Introduction}

Given a graph $G$ on $n$ vertices its cover ideal is the ideal $J(G)=\bigcap (x_i,x_j)\subseteq K[x_1, \ldots ,x_n]$, where the intersection runs over the edges of $G$. The symbolic Rees algebra of this ideal is also known as the vertex cover algebra of $G$. 
In their paper \cite{HHT} Herzog, Hibi and Trung have studied this algebra in the more general context of hypergraphs.  
 In the present paper we study the symbolic fiber cone of $J(G)$, denoted by $\AG$.  The notation is due to Herzog who presented this ring as  the algebra of basic vertex covers of $G$.\\

In the first section of this paper we recall some classical definitions and results from combinatorics and commutative algebra that we will use throughout this work.
In Section 2 we  prove  that for a bipartite graph $G$, the algebra $\AG$ is Koszul. This problem was suggested to us during an informal discussion by J\"urgen Herzog. The Koszul property  follows from  the homogeneous Algebra with Straightening Laws (ASL) structure that we can give to $\AG$. In a joint paper with Benedetti, \cite{BCV}, we gave an equivalent combinatorial condition for $\AG$ being a domain for a bipartite graph. This combinatorial property is called weak square condition (WSC). 
The ASL structure provides in the bipartite case another equivalent condition: $\AG$ is a domain if and only if  $\AG$ is a Hibi ring. 
Using this structure and a result of  Hibi from \cite{H} we are able to characterize for bipartite graphs the Gorenstein domains.
The non-integral case turns out to be  more complicated. However, from   
the description of the poset on which $\AG$ is an ASL we can deduce some nice consequences. For instance, we can produce many examples of bipartite graphs such that $\AG$ is not Cohen-Macaulay, using results of Kalkbrener and Sturmfels \cite{KS} and of the second author \cite{V}. 
With some additional assumption on the combinatorics of the graph  we can prove that $\AG$ is Cohen-Macaulay if and only if it is equidimensional.\\

In the last two sections we will leave the setting of bipartite graphs and  consider general graphs without loops, multiple edges or isolated points. 
In the third section of the paper we compute in terms of the combinatorics of the graph the Krull dimension of $\AG$. The combinatorial invariant that we introduce is called the \emph{graphical dimension}. It turns out it has a lower bound given by the \emph{paired-domination number}  and an upper bound given by the \emph{matching number} of the graph. When the base field is infinite the dimension of  $\AG$ is an upper bound for the arithmetical rank of $J(G)$ localized at the maximal irrelevant ideal, so we get interesting upper bounds for the arithmetical rank of a monomial ideal of pure codimension 2 after having localized at the maximal irrelevant ideal, refining  a result of Lyubeznik \cite{Ly1}.

In the fourth and last section we focus our attention on the edge ideal of the graph, namely $I(G) = (x_ix_j ~:~ \{ij\} \textrm{~is an edge of~}G) \subseteq S = K[x_1,\ldots,x_n]$. Two problems that have recently caught the attention of many authors (see for instance \cite{Fr, HV, HH, Ka, Ku,  Zh}) are the characterization in terms of the combinatorics of $G$ of the Cohen-Macaulay property and the Castelnuovo-Mumford regularity of $S/I(G)$. Our approach is to restrict the problem to a subgraph $\pi(G)$ of $G$ which maintains some useful properties of the edge ideal. This graph is constructed passing through another graph, namely $G^{0-1}$, introduced by Benedetti and the second author in \cite{BV}. Using this tool we are able to extend a result of Herzog and Hibi from \cite{HH} regarding the Cohen-Macaulay property  and a result of Kummini from \cite{Ku} regarding the Castelnuovo-Mumford regularity.\\

The authors wish to thank J\"urgen Herzog for suggesting this topic and for many useful discussions which led to new stimulating  questions and interesting observations. We also wish  to thank Aldo Conca and Bruno Benedetti for their useful comments.

\section{Terminology and Preliminaries}
For the convenience of the reader we include in this short section the standard terminology and the basic facts from combinatorics and commutative algebra which we will use throughout the paper.

\vspace{2mm}

\noindent {\large \textbf{Combinatorics }}\\
For a natural number $n\ge 1$ we denote by $[n]$ the set $\{1,\ldots,n\}$. By a graph $G$  on $[n]$ we understand  a graph with $n$ vertices without loops or multiple edges. If we do not specify otherwise, we also assume that a graph has no isolated points. We denote by $V(G)$ (respectively $E(G)$) the vertex set (respectively the edge set) of $G$. From now on $G$ will always denote a graph on $[n]$ and we write, when it does not raise confusion, just $V$ for $V(G)$ and $E$ for $E(G)$.

We say that a set $M \subseteq E$ of edges is a {\it matching} of $G$ if any two distinct edges of $M$ have empty intersection. A matching is called {\it maximal} if it has maximal cardinality among all  matchings of $G$. The {\it matching number} of $G$, denoted by $\nu(G)$, is the cardinality of a maximal matching of $G$. A matching $M$ is called {\it perfect} if every vertex in $V$ belongs to an edge in $M$. 
A set of vertices $V' \subseteq V$  is called {\it independent} if $\{v,w\} \notin E$ for any $v, w \in V'$. The set $V'$ is called a {\it point cover} of $G$, if for any $w \in V \setminus V'$ there exists a vertex $v \in V'$  such that $\{v,w\} \in E$. A set $E'\subseteq E$ of edges is said to be {\it pairwise disconnected} if for any $e\neq e'\in E'$ we have that $e\cap e'=\emptyset$ and that there is no edge in $E$ connecting $e$ with $e'$.

A set $S \subseteq V$ is called a \emph{paired-dominating set} of $G$ if $S$ is a point cover of $G$  and if the subgraph induced by $S$ has at least one perfect matching. The minimum cardinality of a paired-dominating set is called the \emph{paired-domination number} of $G$ and is denoted by $\gamma_{~\textup{P}}(G)$.\\

A nonzero function $\aa:V(G)\rightarrow \NN$, is a {\it $k$-cover} of $G$ ($k\in \NN$) if $\aa(i)+\aa(j)\geq k$ whenever $\{i,j\}\in E(G)$. A $k$-cover $\aa$ is {\it decomposable} if $\aa=\bb+\cc$  where $\bb$ is an $h$-cover and $\cc$ is a ($k-h$)-cover; $\aa$ is {\it indecomposable} if it is not decomposable. A $k$-cover $\aa$ is called \emph{basic} if it is not decomposable as a $0$-cover plus a $k$-cover (equivalently if no function $\bb < \aa$ is a $k$-cover). The $1$-covers are also known as vertex covers, and basic $1$-covers are the so-called minimal vertex covers.

We recall that a {\it lattice} on a set $L$ is a pair $\L=(L,\prec)$ such that $\prec$ defines a partial order on $L$ for which every two elements $l,l'\in L$ have a supremum, denoted by $l \vee l'$, and a infimum, denoted by $l \wedge l'$. We say that $\L$ is {\it distributive} if \ $l \vee (l' \wedge l'')= (l \vee l')\wedge (l \vee l'')$.

In \cite{BCV}  the authors defined the following property for graphs, which was then proved in \cite{BV} to be equivalent to $\AG$ being a domain. A graph $G$ is said to have the weak square condition (WSC for short) if for every vertex $v \in V$, there exists an edge $\{v,w\} \in E$ containing it such that 
\[ \left. \begin{array}{r}
\{v,v'\} \in E \\
\{w,w'\} \in E
\end{array}
\right\} 
\Rightarrow
\{v',w'\} \in E.
\]

\vspace{2mm}

\noindent {\large \textbf{Commutative Algebra }}\\
Throughout the paper $K$ will be a field, $S=K[x_1, \ldots ,x_n]$ will denote the polynomial ring with $n$ variables over $K$ and $\mm=(x_1,\ldots ,x_n)$ will be the irrelevant maximal ideal of $S$. The {\it edge ideal} of $G$, denoted by $I(G)$, is the square-free monomial ideal of $S$ 
\[I(G)=(x_ix_j: \{i,j\}\in E(G))\subseteq S.\]
A graph $G$ is called {\it Cohen-Macaulay over $K$} if $S/I(G)$ is a Cohen-Macaulay ring. A graph is called just {\it Cohen-Macaulay} if it is Cohen-Macaulay over any field (equivalently over $\mathbb{Z}$). The {\it cover ideal} of $G$ is the Alexander dual of the edge ideal, and we denote it by $J(G)$. So
\[J(G)=\bigcap_{\{i,j\}\in E(G)}(x_i,x_j).\]

As said in the introduction, in this paper we study the symbolic fiber cone of $J(G)$. To introduce it, we recall the definition of the symbolic Rees algebra of an ideal $I\subseteq S$:
\[R(I)_s=\bigoplus_{k\geq 0}I^{(k)}t^k\subseteq S[t],\]
where $I^{(k)}$ denotes the $k$th symbolic power of $I$; i.e. $I^{(k)}=(I^k S_W) \cap S$, where $W$ is the complement in $S$ of the union of the associated primes of $I$ and $S_W$ denotes the localization of $S$ at the multiplicative system $W$. If $I$ is a square-free monomial ideal then $I^{(k)}$ is just the intersection of the (ordinary) $k$th power of the minimal prime ideals of $I$. 
Therefore
\[(J(G))^{(k)}=\bigcap_{\{i,j\}\in E(G)}(x_i,x_j)^k.\]
The symbolic fiber cone of $I$ is $R(I)_s/\mm R(I)_s$.
We will denote by $\AG$ the symbolic fiber cone of $J(G)$.

There is a more combinatorial way to construct   $\AG$, given by the relation between basic  covers and $J(G)$:
\[J(G)^{(k)}=(x_1^{\aa(1)}\cdots x_n^{\aa(n)}: \aa \mbox{ is a basic $k$-cover}).\]
Thus  $R(J(G))_s=K[x_1^{\aa(1)}\cdots x_n^{\aa(n)}t^k: \aa \mbox{ is a $k$-cover}]\subseteq S[t]$.
For more details on this interpretation of these algebras see \cite{HHT}, in which this symbolic Rees algebra is denoted by $A(G)$. The authors of that paper proved many properties of $A(G)$. First of all they noticed that $A(G)$ is a finitely generated $K$-algebra, since it is generated in degree less than or equal to $2$. Moreover $A(G)$ is a standard graded $S$-algebra  if and only if $G$ is bipartite. They also proved that $A(G)$ is always a Gorenstein normal domain.

Since $\AG = A(G)/ \mm A(G)$, we have that
\[\AG=K \oplus \Big(\bigoplus_{k\geq 1}<x_1^{\aa(1)}\cdots x_n^{\aa(n)}t^k: \aa \mbox{ is a basic $k$-cover}>\Big),\]
where the  multiplication table is given by 
\begin{displaymath}
x_1^{\aa(1)}\cdots x_n^{\aa(n)}t^k \cdot x_1^{\bb(1)}\cdots x_n^{\bb(n)}t^h = \left\{ \begin{array}{ll} x_1^{\cc(1)}\cdots x_n^{\cc(n)}t^{h+k} & \mbox{if $\cc=\aa +\bb$ is a basic $(h+k)$-cover},\\
0 &  \mbox{otherwise. } \end{array}  \right.
\end{displaymath}
With the above presentation it is clear that the Hilbert function of $\AG$ counts the basic $k$-covers of $G$, i.e.
\[\HF(k):=\dim_K (\AG_k)=|\{\mbox{basic $k$-covers of }G\}|.\]
It turns out that the number of basic $2h$-covers of a graph grows as a polynomial in $h$ of degree $\dim \AG - 1$, namely  the Hilbert polynomial $\HP$ of the second Veronese of $\AG$, which is standard graded (see Remark \ref{polynomial}). This simple fact will be the main tool in the characterization of the Krull  dimension of $\AG$ in terms of $G$. 

From the above discussion  it follows that $\AG$ is a standard graded $K$-algebra (equivalently it is the ordinary fiber cone of $J(G)$) if and only if $G$ is bipartite. The graphs for which $\AG$ is a domain have been characterized in  \cite{BCV} in the bipartite case and in \cite{BV} in general. Moreover, if $\AG$ is a domain then it is Cohen-Macaulay, but it may not be Gorenstein. When $G$ is bipartite and  $\AG$ is not a domain the projective scheme defined by $\AG$ is connected, but not necessarily equidimensional, and therefore it may be non-Cohen-Macaulay (for more details see \cite{BCV}).

\section{Koszul Property and ASL structure of $\AG$}

During an informal conversation at Oberwolfach J\"urgen Herzog asked whether $\AG$ is Koszul provided that $G$ is bipartite. In this section we  answer  this question positively, showing even more: if $G$ is bipartite, then $\AG$ has a structure of homogeneous ASL.

Algebras with straightening laws (ASL's for short) were introduced by  De Concini,  Eisenbud and  Procesi in \cite{DCEP}.  
These algebras provide an unified treatment of both algebraic and geometric objects that have a combinatorial nature. For example, the coordinate rings of some classical algebraic varieties (such as determinantal rings and Pfaffian rings) have an ASL structure. For more details on this topic the reader can consult the book of  Bruns and  Vetter \cite{BrVe}. 
First, we will recall the definition of homogeneous ASL on posets.

Let $(P,<)$ be a finite poset and denote by $K[P]$ the polynomial ring whose variables are the elements of $P$. Denote by $I_P$ the following  monomial ideal of $K[P]$:
\[ I_P = ( xy~:~x \textrm{~and~} y \textrm{~are incomparable elements of~}P).\]
\begin{definition}
Let $A = K[P]/I$, where $I$ is a homogeneous ideal with respect to the usual grading. The graded algebra $A$ is called a homogeneous ASL on $P$ if
\begin{compactitem}
\item[(ASL1)] The residue classes of the monomials not in $I_P$ are linearly independent in $A$.
\item[(ASL2)] For every $x,y \in P$ such that $x$ and $y$ are incomparable the ideal $I$ contains a polynomial of the form
\[ xy - \sum\lambda zt\]
with $\lambda \in K$, $z,t,\in P$, $z\le t$, $z<x$ and $z<y$.
\end{compactitem}
\end{definition}
The polynomials in (ASL2) give a way of rewriting in $A$ the product of two incomparable elements. These relations are called the \emph{straightening relations} or straightening laws.

A total order $<'$ on $P$ is called  a linear extension of the poset $(P,<)$ if $x<y$ implies $x<'y$. It is known that if $\tau$ is a revlex term order with respect to a linear extension of $<$, then the polynomials in (ASL2) form a \GB~ of $I$ and $\ini_{\tau}(I)=I_P$. \\

We will prove now that when $G$ is a bipartite graph, $\AG$ has an ASL structure. Let us first fix some notation.
Let $G$ be a bipartite graph with the partition of the vertex set $[n]=A\cup B$ and suppose that $|A|\leq |B|$. We denote by$\P(G)$  the set of basic 1-covers of $G$ and we define on this set the following partial order
\[ \aa \leq \bb \ \ \ \iff \ \ \ \aa(a)\leq \bb(a) \ \ \forall \ \ a\in A. \]
Given two basic 1-covers $\aa$ and $\bb$ of $G$, we can define the following $1$-covers:
\begin{displaymath}
(\aa \tw \bb)(v) = \left\{ \begin{array}{lc} \min\{\aa(v),\bb(v)\} & \mbox{if } v \in A.\\
\max\{\aa(v),\bb(v)\} &  \mbox{if } v\in B; \end{array}  \right.
\end{displaymath}
\begin{displaymath}
(\aa \tv\bb)(v) = \left\{ \begin{array}{lc} \max\{\aa(v),\bb(v)\} & \mbox{if } v \in A,\\
\min\{\aa(v),\bb(v)\} &  \mbox{if } v\in B. \end{array}  \right.
\end{displaymath}
Notice that these $1$-covers may be {\it non-basic}. It is easy to check the following equality:
\[ \aa + \bb = \aa \tw \bb + \aa \tv \bb . \]
The above equality translates to a relation among the generators of $\AG$ in the following way. Denote by $R= K[\P(G)]$. In order to simplify notation we will denote by $\aa,\bb,\cc, \ldots$ both the variables of $R$ and the basic 1-covers of $G$. Whenever it will not be clear from the context, we will specify which of the two we are considering. We have the following  natural presentation of $\AG$:

$$ \begin{array}{rcllc} 
\Phi :& R & \longrightarrow & \AG &  \\
&{\aa} & \longmapsto & x_1^{\aa(1)} \cdots x_n^{\aa(n)}t & 
\end{array} $$
For simplicity we set $\aa \tw \bb$ (respectively $\aa \tv \bb$) to be $0$ (as elements of $R$) whenever they are not basic 1-covers.
Using this convention  it is obvious that the polynomial $\aa\bb - (\aa \tw \bb)(\aa \tv \bb)$ belongs to the kernel of $\Phi$ for any pair of basic $1$-covers $\aa$ and $\bb$. \\
The main result of this section is the following theorem.

\begin{thm}\label{ASL}
Let $G$ be a bipartite graph. The algebra $\AG$ has a homogeneous ASL structure on $\P(G)$ over $K$. With the above notation, the straightening relations are
\begin{displaymath}
\Phi(\aa)\Phi(\bb)= \left\{ \begin{array}{ll}  \Phi(\aa \tw \bb)\Phi(\aa \tv \bb)& \mbox{if both $\aa \tw \bb$ and $\aa \tv \bb$ are basic $1$-covers}, \\
0 & \mbox{otherwise;} \end{array} \right. 
\end{displaymath} 
for any $\aa$ and $\bb$ incomparable basic $1$-covers. 
In particular we have
\[ \ker \Phi = (\aa\bb - (\aa \tw \bb)(\aa \tv \bb)~:~ \aa \mbox{ and $\bb$ are incomparable basic $1$-covers}). \]
\end{thm}

Before proving Theorem \ref{ASL}, we will prove the following lemma.

\begin{lemma}\label{standard}
Let $G$ be a bipartite graph. For every $i=1,...,k$ let $\aa_1^i \leq \ldots \leq \aa_d^i$ be (distinct) $d$-multi-chains of basic $1$-covers, and set $U_i=\aa_1^i\cdots \aa_d^i \in R$.
If a linear combination with coefficients in $K$, say  $F = \lambda_1 U_1+\ldots + \lambda_k U_k$,  belongs to the kernel of $\Phi$, then $\lambda_1=\ldots =\lambda_k=0$. 
\end{lemma}
\begin{proof}
We will prove the lemma by induction on $k$.

Let $k=1$. 
Denote $\aa_j^1$ by $\aa_j$ and $U_1$ by $U$. Suppose $\lambda U \in \ker \Phi$, with $0 \neq \lambda \in K$. As $\Phi(\lambda U) = \lambda \Phi(U)$, this implies that $\Phi(U)=0$. In other words  the $d$-cover $\cc$ that associates to a vertex $v$ the value $\cc(v)=\aa_1(v)+\ldots +\aa_d(v)$ is non-basic. So there exists a vertex $v_0$ of $G$ such that $\gamma(v_0)+\gamma(w)>d$ for any $w$ adjacent to $v_0$. We may  assume that $v_0\in A$ (otherwise the issue is symmetric).\\
 Set  $q=\min \{i=1,\ldots ,d: \aa_i(v_0)=1\}$; if $\aa_i(v_0)=0$ for any $i$ we set $q=d+1$.
Since $\aa_q$ is a basic $1$-cover, there exists a vertex $w_0$ adjacent to $v_0$ such that $\aa_q(v_0)+\aa_q(w_0)=1$. As $\aa_1 \leq \ldots \leq \aa_d$, we have $\aa_i(v_0)=0$ for any $i\leq q$, and $\aa_j(w_0)=0$ for any $j\geq q$ (because $w_0\in B$). This implies that 
\[ \gamma(v_0)+\gamma(w_0)=\sum_{i=q}^d \aa_i(v_0)+ \sum_{j=1}^{q-1}\aa_j(w_0) = (d-q+1)+(q-1)=d, \]
a contradiction.

Let $k > 1$. Since the multi-chains are distinct, there exists an index $j=1,\ldots ,d$ such that $\aa_j^r \neq \aa_j^s$ for some $r \neq s$. So there exists a vertex $v_0$ of $G$ and a partition of $[k] = \mathcal{I} \cup ( [k] \setminus \mathcal{I})$, with  $\emptyset \neq \mathcal{I} \neq [k]$,
such that 
\[
\aa_j^i(v_0)= \left\{ \begin{array}{ll} 0 & \textrm{~if~~} i \in \mathcal{I}\\
1 & \textrm{~if~~} i \in [k]\setminus \mathcal{I}. \end{array} \right. 
\]
We can again assume that $v_0\in A$ and, up to a relabeling, that $v_0=1\in [n]$. 
Since we are dealing with chains, we have that $\aa_s^i(1)=0$ for every $i\in \mathcal{I}$,and $s\leq j$. For the multi-chains indexed by $[k]\setminus \mathcal{I}$ we have $\aa_t^h(1)=1$ for every $h\in[k]\setminus \mathcal{I}$, and $t\geq j$. This means that we can rewrite the equation $\Phi(F)=0$ as
\[ x_1^{d-j+1}\displaystyle \left( \sum_{h\in[k]\setminus \mathcal{I}}\lambda_h \frac{\Phi(U_h)}{x_1^{d-j+1}}\right)  \ + \ \sum_{i \in  \mathcal{I}}\lambda_i \Phi(U_i) \ = \ 0 , \]
where $\sum_{i \in \mathcal{I}}\lambda_i \Phi(U_i)$ has at most degree $d-j$, with respect to $x_1$. This implies that 
\[ \sum_{h\in [k]\setminus \mathcal{I}}\lambda_h \Phi(U_h)=\sum_{i\in  \mathcal{I}}\lambda_i \Phi(U_i) = 0,\]
so we can conclude by induction.
\end{proof}

\begin{proof}[Proof of Theorem \ref{ASL}]
 We have seen that $\AG = R/ \ker\Phi$.
 Because $G$ is bipartite, the graded $K$-algebra $\AG$ is generated by the elements $x^{\aa}$, with $\aa$  a basic $1$-cover. Moreover the degree of $x^{\aa}$ is $1$ if $\aa$ is a basic $1$-cover. So  $\ker\Phi$ is homogeneous with respect to the standard grading of $R$ . We need to see now that (ASL1) and (ASL2) are satisfied.

The first condition  follows by Lemma \ref{standard}. From the discussion preceding  Theorem \ref{ASL} we get that the polynomials  $\aa\bb - (\aa \tw \bb)(\aa \tv \bb)$ belong to $\ker\Phi$. It is easy to see that by construction $(\aa \tw \bb) \le (\aa \tv \bb)$, $(\aa \tw \bb) < \aa$ and $(\aa \tw \bb) < \bb$ hold whenever  $ \aa \tw \bb$ and $\aa \tv \bb$ are basic 1-covers. So (ASL2) holds as well.
The last part of the statement  follows immediately from \cite[Proposition 4.2]{BrVe}.
\end{proof}
As we said in the beginning of this section, the homogeneous ASL structure of $\AG$ implies that the straightening relations form a quadratic  \GB. This implies the following corollary.
\begin{corollary}\label{koszul}
If $G$ is a bipartite graph, then $\AG$ is a Koszul algebra.
\end{corollary}

\begin{remark}
Independently and by different methods Rinaldo showed in \cite[Corollary 3.9]{Ri} a particular case of Corollary \ref{koszul}. Namely he proved that $\AG$ is Koszul provided that $G$ is a bipartite graph satisfying the WSC. Actually we will show in Corollary \ref{Hibi} that for such a graph $\AG$ is even a Hibi ring.
\end{remark}

A special class of algebras with straightening laws are the so called Hibi rings. They were constructed in \cite{H} as an example of integral ASLs. The poset that supports their structure is a distributive lattice $\L$ and the straightening relations are given for any two incomparable elements $p, q \in \L$ by
\[ pq - (p\wedge q)(p\vee q),\]
where $p\wedge q$ denotes the infimum and $p \vee q$ the supremum of $p$ and $q$.

We will see that there exists a correspondence between the vertex covers $\aa \tw \bb$ (resp. $\aa \tv \bb$) and the infimum (resp. supremum) of $\aa$ and $\bb$ in the poset $\P(G)$.
\begin{remark}\label{inf}
Let $\aa, \bb \in \P(G)$ be two incomparable basic 1-covers.
\begin{compactitem}
\item[1.]  If $\aa \tw \bb$  is a basic 1-cover then $\aa \tw \bb = \aa \wedge \bb$.
\item[2.] If $\aa \tv \bb$ is a basic 1-cover then $\aa \tv \bb = \aa \vee \bb$.
\end{compactitem}
\end{remark}

\begin{proof}
To prove 1. we have to show that if $\cc \in \P(G)$ with $\cc < \aa$ and $\cc < \bb$ then $ \cc \le  \aa \tw \bb$. This means that $\cc(a) \le \aa(a)$ and $\cc(a) \le \bb(a)$ for every $a \in A$. So $\cc(a) \le \min\{\aa(a),\bb(a)\}$ for every $a \in A$ and we are done. 
The second part is proved analogously.
\end{proof}
Using this remark we are able to prove the following corollary.

\begin{corollary}\label{Hibi}
Let $G$ be a bipartite graph. The following are equivalent:
\begin{compactitem}
\item[1.] $G$ satisfies the WSC;
\item[2.] $\AG$ is a domain;
\item[3.] $\AG$ is a Hibi ring on $\P(G)$ over $K$.
\end{compactitem}
\end{corollary}
\begin{proof}
The equivalence between 1. and 2. was already proved in \cite[Theorem 1.9]{BCV} and we present it here only for completeness. The fact that 3. implies 2.  was proved by  Hibi in the same paper where he introduced these algebras (see \cite[p. 100]{H}). So we only need to prove that 2. implies 3.

When $\AG$ is a domain, for every $\aa, \bb \in \P(G)$ that are incomparable, we must have $\aa\bb =   (\aa \tw \bb)(\aa \tv \bb)$. This means that both   $\aa \tw \bb$ and $\aa \tv \bb$ are basic 1-covers, so by Remark \ref{inf} they coincide with $\aa \wedge \bb$, respectively with $\aa \vee\bb$. In other words the poset $\P(G)$ is a  lattice. So by \cite[p.100]{H} and by Theorem \ref{ASL} we conclude. 
\end{proof}

A classical structure theorem of  Birkhoff \cite[p.59]{B} states that for each distributive lattice $\L$ there exists a unique poset $P$ such that $\L = J(P)$, where $J(P)$ is the set of poset ideals of $P$, ordered by inclusion. By Corollary \ref{Hibi} we have that if a bipartite graph $G$ satisfies the WSC, then the poset of basic 1-covers $\P(G)$ is a distributive lattice. So by Birkhoff's result there exists a unique poset $P_G$ such that $\P(G) = J(P_G)$. We  use now another result of Hibi which describes completely the Gorenstein Hibi rings (see \cite[p.105]{H}) to obtain the following corollary.

\begin{corollary}
Let $G$ be a bipartite graph satisfying the WSC. The following conditions are equivalent:
\begin{compactitem}
\item[1.] $\AG$ is Gorenstein;
\item[2.] the poset $P_G$ defined above is pure.
\end{compactitem}
\end{corollary}

We want to close this section showing some tools to deduce properties of $\AG$ by the combinatorics of $\P(G)$. In particular we will focus on the Cohen-Macaulayness of $\AG$, but one can read off by $\P(G)$ also the dimension, the multiplicity, and the Hilbert series of $\AG$.

The main technique is to consider the ``canonical" initial ideal of the ideal defining $\AG$. 
Let $I$ be the ideal,  which we described above in terms of its  generators, such that $\AG =R/I$ (recall that $R= K[\P(G)])$. Denote by $\mbox{in}(I)$ the initial ideal of it with respect to a degrevlex term order associated to  a linear extension of $(\P(G),<)$. From the results of this section it follows that $\mbox{in}(I)$ is a square-free monomial ideal, so we can associate to it a simplicial complex $\Delta=\Delta(\mbox{in}(I))$. Moreover it is easy to show that $\Delta$ is the ordered complex of $\P(G)$, i.e. its faces are the chains of $\P(G)$. 

\begin{example} \emph{$\AG$ non Cohen-Macaulay}.  
Let $G$ be a path of length $n-1\geq 5$. So $G$ is a graph on $n$ vertices with edges:
\[ \{1,2\},\ \{2,3\}, \ \ldots , \ \{n-1,n\}. \]
For any $i=1,\ldots , \lfloor n/2 \rfloor$ define the basic $1$-cover
\begin{displaymath}
\aa_i(j) = \left\{ \begin{array}{cl} 1 &~~~\mbox{if \ } j=2k \mbox{ \ and \ } k\leq i, \\
 1 &~~~\mbox{if \ $j=2k-1$ \ and \ }k>i, \\
 0 &~~~\textrm{otherwise}. \end{array}  \right.
\end{displaymath}
Then define also the basic $1$-cover
\begin{displaymath}
\bb(j) = \left\{ \begin{array}{cl} 1 &~~~\mbox{if } j=1,3 \textrm{~~or~~} j = 2k, \textrm{~with~} k \ge 2,  \\
0 &~~~\textrm{otherwise}. \end{array}  \right.
\end{displaymath}
It is straightforward to verify that $\aa_1\leq \aa_2\leq \aa_3\leq \ldots \leq \aa_{\lfloor n/2 \rfloor}$ and $\bb \leq$ $ \aa_3\leq $ $ \ldots \leq $ $ \aa_{\lfloor n/2 \rfloor}$ are maximal chains of $\P(G)$. So $\P(G)$ is not pure. Therefore the ordered complex of $\P(G)$ is not pure. So $\AG$ is not an equidimensional ring by \cite[Corollary 1]{KS}. In particular, if $G$ is a path of length at least $5$, $\AG$ is not Cohen-Macaulay.
\end{example}

Before stating the following result we recall some notion regarding posets. A poset $P$ is \emph{bounded} if it has a least and a greatest element. An element $x\in P$ \emph{covers} $y\in P$ if $y\leq x$ and there not exists $z\in P$ with $y<z<x$. The poset $P$ is said to be \emph{locally upper semimodular} if whenever $v_1$  and $v_2$ cover $u$ and $v_1,v_2 < v$ for some $v$ in $P$, then there exists $t\in P$, $t\leq v$, which covers $v_1$ and $v_2$.

\begin{thm}\label{loc}
Let $G$ be a bipartite graph and $A\cup B$ a bipartition of the vertex set with $|A|\leq |B|$. Moreover, let $\Delta$ be the ordered complex of $\P(G)$. If $\rank(\P(G))=|A|$, then the following are equivalent:
\begin{compactenum}
\item $\AG$ is equidimensional;
\item $\P(G)$ is a pure poset;
\item $\Delta$ is shellable;
\item $\AG$ is Cohen-Macaulay.
\end{compactenum}
\end{thm}
\begin{proof}
\mbox{$4.\Rightarrow1.$} is well known. 
As the Cohen-Macaulayness of $R/\mbox{in}(I)$ implies the Cohen-Macaulayness of $R/I\cong \AG$, 
\mbox{$3.\Rightarrow4.$}  is also true. 
\mbox{$1.\Rightarrow2.$}  follows by \cite[Corollary 1]{KS}.

\mbox{$2.\Rightarrow3.$}  To prove that $\Delta$ is shellable we will use a result of Bj\"orner (see \cite[Theorem 6.1]{Bj}), stating  that it is enough to show  that $\P(G)$ is a bounded locally upper semimodular poset. The poset $\P(G)$ is obviously  bounded,
so let $\aa$ and $\bb$ be two elements of $\P(G)$ which cover $\cc$. 
The fact that   $\rank(\P(G))=|A|$ 
 together with the pureness of $\P(G)$, imply that for a basic $1$-cover $\xi$ we have $\rank(\xi)=\sum_{v\in A}\xi(v)$. If $\aa$ and $\bb$ cover $\gamma$, since all the unrefinable chains between two incomparable elements of a bounded pure complex have the same length, it follows that $s=\rank(\aa)=\rank(\bb)=\rank(\cc)+1$. But $\cc(v)\leq \min \{\aa(v),\bb(v)\}$, for each $v\in A$, so if we look at the rank of the elements involved we obtain $\cc(v)=\min \{\aa(v),\bb(v)\}$ for all $v\in A$. 
  Consider the (non necessarily basic)  $1$-cover, defined at the beginning of this section: $\aa \tv \bb$. It is easy to see that, to make it basic, we can reduce its value at some vertex in $B$, and not in $A$. Let $\delta$ be the basic $1$-cover obtained  from $\aa \tv \bb$. Then 
\[ \rank(\delta)=\sum_{v\in A}\delta(v)=\sum_{v\in A}(\aa \tv \bb)(v)=s+1, \]
which implies that $\delta$ covers $\aa$ and $\bb$. 
\end{proof}
We will see in the next section, that the dimension of $\AG$ is equal to a combinatorial invariant of the graph $G$, called graphical dimension ($\gdim$). In particular we will have that $\rank(\P(G)) = \gdim(G) -1$, so the hypothesis of the theorem regards just the combinatorics of the graph.\\

We showed in \cite{BCV} that $\AG$ domain implies $\AG$ Cohen-Macaulay. Given the above example and theorem it is natural to ask the following questions: ``Can $\AG$ be Cohen-Macaulay and not a domain?". ``Are there examples of graphs for which $\P(G)$ is pure but $\AG$ is not Cohen-Macaulay?". Both answers are  positive and they are provided by the following examples.

\begin{example}
\begin{compactenum}
\item \emph{$\AG$ Cohen-Macaulay but not domain.}
Consider the graph $G$ on seven  vertices below:

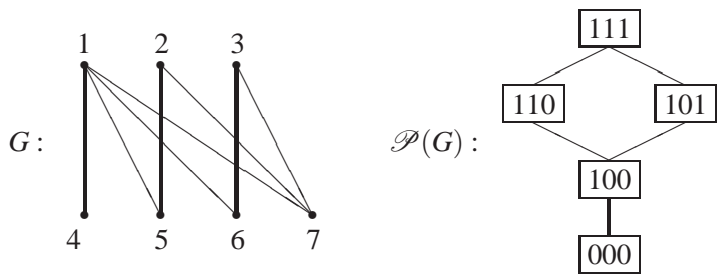
\[
\setlength{\unitlength}{1mm}
\begin{picture}(100,50)
\multiput(0,11.5)(10,0){4}{\circle*{1}}
\multiput(0,31.5)(10,0){3}{\circle*{1}}

\multiput(0,31.5)(20,0){2}{\line(1,-2){10}}
\put(0,31.5){\line(3,-2){30}}
\put(0,31.5){\line(1,-1){20}}
\put(10,31.5){\line(1,-1){20}}
\linethickness{0.4mm}
\multiput(0,11.5)(10,0){3}{\line(0,1){20}}

{\small
\put(-10,20){$G:$}
\put(40,20){$\P(G):$}
\put(65,5){\fbox{000}}
\put(65,15){\fbox{100}}

\put(55,25){\fbox{110}}
\put(75,25){\fbox{101}}
\put(65,35){\fbox{111}}

\put(-2.4,7){4}
\put(9.2,7){5}
\put(19.2,7){6}
\put(29.2,7){7}
\put(-1,33){1}
\put(9.2,33){2}
\put(19.2,33){3}
}

\put(69,8.8){\line(0,1){5}}
\put(69,18.8){\line(-2,1){10}}
\put(69,18.8){\line(2,1){10}}
\put(69,33.83){\line(2,-1){10}}
\put(69,33.83){\line(-2,-1){10}}
 \end{picture}
\]

The vertical edges are the only right edges of $G$, so as $7$ is not contained in any of them,  $G$ does not satisfy the WSC. We order the basic 1-covers component-wise  with respect to the values they take on the vertex set \{1,2,3\}. It is clear from the Hasse diagram above that  $\P(G)$  is pure. Moreover $\rank(\P(G))=3$, so Theorem \ref{loc} implies that $\AG$ is Cohen-Macaulay.\\
\item \emph{$\P(G)$ pure but $\AG$ not Cohen-Macaulay.}
Consider the graph $G$ in the picture below. It is not difficult to see that it has only six basic 1-covers. On the right you can see the Hasse diagram of the poset $\P(G)$. The values written next to the vertices represent the basic 1-cover written in bold on the right. Notice that the partial order is defined component-wise with respect to the values taken on the ``upper" vertices of $G$. 
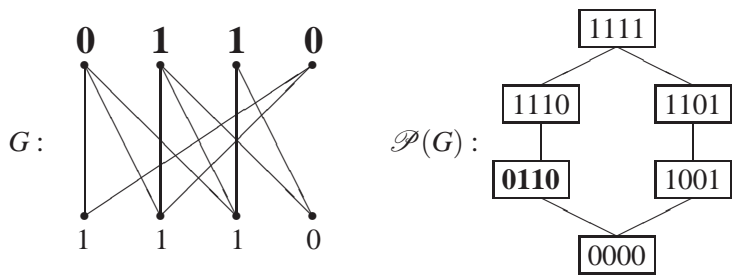
\[
\setlength{\unitlength}{1mm}
\begin{picture}(100,50)
\multiput(0,11.5)(10,0){4}{\circle*{1}}
\multiput(0,31.5)(10,0){4}{\circle*{1}}

\multiput(0,11.5)(10,0){3}{\line(0,1){20}}
\multiput(0,31.5)(10,0){3}{\line(1,-2){10}}
\multiput(0,31.5)(10,0){2}{\line(1,-1){20}}
\put(0,11.5){\line(3,2){30}}
\put(10,11.5){\line(1,1){20}}

{\small
\put(-10,20){$G:$}
\put(40,20){$\P(G):$}
\put(65,5){\fbox{0000}}
\put(53.7,15){{\fbox{$\bold{0110}$}}}

\put(75,15){\fbox{1001}}
\put(55,25){\fbox{1110}}
\put(75,25){\fbox{1101}}
\put(65,35){\fbox{1111}}

\put(-2.4,7){\phantom{(}1}
\put(9.2,7){1}
\put(19.2,7){1}
\put(29.2,7){0}
{\large\put(-1.1,33){$\bold{0}$}
\put(8.8,33){$\bold{1}$}
\put(18.8,33){$\bold{1}$}
\put(28.9,33){$\bold{0}$}
}
}
\put(70,8.8){\line(2,1){10}}
\put(70,8.8){\line(-2,1){10}}
\put(60,18.8){\line(0,1){5}}
\put(80,18.8){\line(0,1){5}}
\put(70,33.83){\line(2,-1){10}}
\put(70,33.83){\line(-2,-1){10}}
 \end{picture}
\]
The poset $\P(G)$ is pure, but the ordered complex of it is not strongly connected. Then $I$ has an initial ideal not connected in codimension $1$, so  \cite[Corollary 2.13]{V} implies that $\AG$ is not Cohen-Macaulay.
\end{compactenum}
\end{example}

\section{The Krull Dimension of $\AG$}

In this section we will extend the notion of graphical dimension, introduced for bipartite graphs in \cite{BCV}, to general graphs.
In \cite{BCV}, together with Benedetti, we conjectured that for a bipartite graph, the Krull dimension of $\AG$ is equal to the graphical dimension of $G$. We will prove that this is true not only in the case of bipartite graphs, but for any graph $G$, considering the extension of the graphical dimension  given in Definition \ref{gdim}. As consequences of this result we are able to give interesting upper bounds for the arithmetical rank of monomial ideals of pure codimension $2$ in $S_{\mm}$, refining  in this case an upper bound given in \cite{Ly1}.\\

For a graph $G$, let $\{a_1, \ldots ,a_r\}=A \subseteq V$ be a nonempty set of independent vertices. We say that $A $ is a \emph{free parameter} set if 
there exists a set of vertices $B = \{b_1, \ldots, b_r\} \subseteq V$,  with $A \cap B = \emptyset$  and the property that:
\begin{compactitem}
\item[-] $\{a_i,b_i\} \in E$ for all $1 \le i \le r$,
\item[-] $\{a_i,b_j\} \in E$ implies $i\le j$.
\end{compactitem}
We will call such a set $B$ a \emph{partner set} of $A$.
\begin{definition}\label{gdim}
Let $G$ be a graph. We define the \emph{graphical dimension} of $G$ as:
\[\gdim(G) := \max\{ |A| ~:~ A \subseteq V \textup{~ is a free parameter set}\} +1.\]
\end{definition}
\begin{remark}
\begin{compactenum}
\item Being free parameter set depends on the labeling of the vertices in both $A$ and $B$.  
\item In the case of bipartite graphs it is not difficult to verify that  this definition coincides with the definition given in \cite{BCV}.
\item In general, $B$ may not be a set of independent vertices (i.e. there may be edges connecting two vertices of $B$).
\end{compactenum}
\end{remark}
The graphical dimension of a graph is not always easy to compute and we were not able to express it in terms of classical invariants of graphs in general. In the following example we will see that the graphical dimension does not depend on the local degree of the vertices. By local degree of a vertex we understand the number of edges incident in that vertex. 
\begin{example}
Let $G$ and $G'$ be the bipartite graphs  represented below.
If $V(G) = A \cup B$ and $V(G') = A' \cup B'$  it turns out that all four sets have two vertices of local degree 2 and two vertices of local degree 3. However, we have $\gdim(G) = 2$ and $\gdim(G') =3$.
\[
\setlength{\unitlength}{2mm}
\begin{picture}(60,20)
\put(0, 9.5){$G : $}
\put(7.5,10){\circle*{0.75}}  
\put(12.5,10){\circle*{0.75}}
\put(15,5){\circle{0.75}}
\put(15,15){\circle{0.75}}
\put(20,5){\circle*{0.75}}
\put(20,15){\circle*{0.75}}
\put(22.5,10){\circle{0.75}}
\put(27.5,10){\circle{0.75}}

\put (7.5,10){\line(3,-2){7.22}}
\put (7.5,10){\line(3,2){7.22}}
\put (12.5,10){\line(1,-2){2.33}}
\put (12.5,10){\line(1,2){2.33}}
\put (20,5){\line(-1,0){4.625}}
\put (20,5){\line(1,2){2.35}}
\put (20,5){\line(3,2){7.23}}
\put (20,15){\line(-1,0){4.625}}
\put (20,15){\line(1,-2){2.35}}
\put (20,15){\line(3,-2){7.23}}

\put(40,5){\circle{0.75}}
\put(45,5){\circle{0.75}}
\put(50,5){\circle{0.75}}
\put(55,5){\circle{0.75}}
\put(40,15){\circle*{0.75}}
\put(45,15){\circle*{0.75}}
\put(50,15){\circle*{0.75}}
\put(55,15){\circle*{0.75}}

\put(40,15){\line(0,-1){9.625}}
\put(40,15){\line(1,-2){4.8}}
\put(40,15){\line(1,-1){9.7}}
\put(45,15){\line(0,-1){9.625}}
\put(45,15){\line(1,-2){4.8}}
\put(45,15){\line(1,-1){9.7}}
\put(50,15){\line(1,-2){4.8}}
\put(50,15){\line(-1,-1){9.7}}
\put(55,15){\line(0,-1){9.625}}
\put(55,15){\line(-3,-2){14.625}}

{\footnotesize
\put(6,9.5){3}
\put(11,9.5){4}
\put(19.5,16){1}
\put(19.5,3){2}
\put(14.5,16){$a$}
\put(14.5,3){$d$}
\put(24,9.5){$b$}
\put(29,9.5){$c$}

\put(49.5,16){3}
\put(54.5,16){4}
\put(39.5,16){1}
\put(44.5,16){2}
\put(39.5,3){$a$}
\put(44.5,3){$b$}
\put(49.5,3){$c$}
\put(54.5,3){$d$}
}

\end{picture}
\]

\[
\setlength{\unitlength}{2mm}
\begin{picture}(60,20)
\put(0, 9.5){$G' : $}
\put(12.5,20){\circle*{0.75}}  
\put(12.5,10){\circle{0.75}}
\put(15,5){\circle*{0.75}}
\put(15,15){\circle*{0.75}}
\put(20,5){\circle{0.75}}
\put(20,15){\circle{0.75}}
\put(22.5,10){\circle*{0.75}}
\put(22.5,20){\circle{0.75}}

\put (12.5,20){\line(0,-1){9.625}}
\put (12.5,20){\line(3,-2){7.15}}
\put (12.5,20){\line(1,0){9.625}}
\put (15,5){\line(1,0){4.625}}
\put (15,5){\line(-1,2){2.35}}
\put (15,15){\line(-1,-2){2.35}}
\put (15,15){\line(1,0){4.625}}
\put (22.5,10){\line(-1,2){2.35}}
\put (22.5,10){\line(-1,-2){2.35}}
\put(22.5,10){\line(0,1){9.625}}

\put(40,5){\circle{0.75}}
\put(45,5){\circle{0.75}}
\put(50,5){\circle{0.75}}
\put(55,5){\circle{0.75}}
\put(40,15){\circle*{0.75}}
\put(45,15){\circle*{0.75}}
\put(50,15){\circle*{0.75}}
\put(55,15){\circle*{0.75}}

\put(40,15){\line(0,-1){9.625}}
\put(40,15){\line(1,-1){9.7}}
\put(45,15){\line(0,-1){9.625}}
\put(45,15){\line(1,-2){4.8}}
\put(45,15){\line(1,-1){9.7}}
\put(50,15){\line(0,-1){9.625}}
\put(50,15){\line(1,-2){4.8}}
\put(55,15){\line(0,-1){9.625}}
\put(55,15){\line(-1,-1){9.7}}
\put(55,15){\line(-3,-2){14.625}}

{\footnotesize
\put(11,19.5){2}
\put(11,9.5){$c$}
\put(19.5,16){$d$}
\put(19.5,3){$a$}
\put(14.5,16){3}
\put(14.5,3){1}
\put(24,9.5){4}
\put(24,19.5){$b$}

\put(49.5,16){3}
\put(54.5,16){4}
\put(39.5,16){1}
\put(44.5,16){2}
\put(39.5,3){$a$}
\put(44.5,3){$b$}
\put(49.5,3){$c$}
\put(54.5,3){$d$}
}

\end{picture}
\]
For $G$ a free parameter set of maximal cardinality is $\{1,2\}$ with partner set $\{a,b\}$. For $G'$ we have $\{1,2,3\}$ with partner set $\{a,b,c\}$. In general these sets are not unique. For instance, another free parameter set of maximal size for $G$ is $\{2,3\}$ with partner set $\{b,d\}$.
\end{example}

\begin{remark}\label{bounds}
We recall that for a graph $G$ we denote the paired domination number by $\gamma_{~\textup{P}}(G)$ and the matching number by $\nu(G)$. The graphical dimension is bounded by these two numbers in the following way:
\[ \frac{\gamma_{~\textup{P}}(G)}{2}+1 \leq \gdim(G) \leq \nu(G)+1. \]
The second inequality is straightforward from the definition. The first is easy too. To see it, suppose that $A=\{a_1,\ldots ,a_r\}$ is a free parameter set with partner set $B=\{b_1,\ldots ,b_r\}$. If $\gamma_{~\textup{P}}(G)>2r$, then there is a vertex $v$ 
in $V \setminus (A \cup B)$ adjacent to none of the vertices of $A\cup B$. Choose a vertex $w$ adjacent to $v$, and set $a_{r+1}=v$, $b_{r+1}=w$. It turns out that $\{a_1,\ldots ,a_r,a_{r+1}\}$ is a free parameter set with partner set $\{b_1,\ldots ,b_r,b_{r+1}\}$.
\end{remark}

\begin{example}
In this example we will see that the graphical dimension may reach both the upper and lower bound given in the previous remark. The thick lines in the pictures on the left represent the edges of a minimal  paired dominating set.
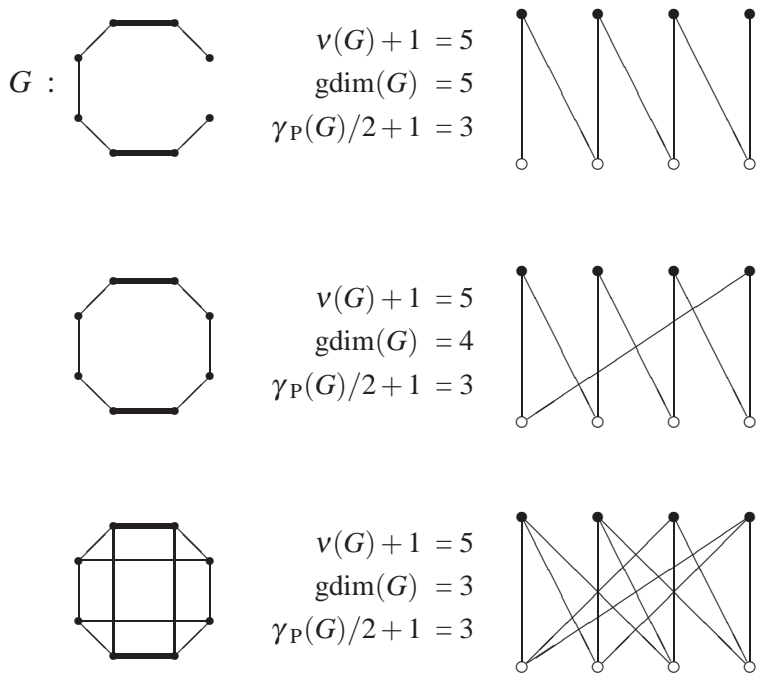
\[
{\setlength{\unitlength}{1.15mm}
\begin{picture}(25,25)

\put(5,14){\circle*{1}}
\put(5,21){\circle*{1}}
\put(9,10){\circle*{1}}
\put(9,25){\circle*{1}}
\put(16,10){\circle*{1}}
\put(16,25){\circle*{1}}
\put(20,14){\circle*{1}}
\put(20,21){\circle*{1}}

\put(5,14){\line(0,1){7}}
\put(5,21){\line(1,1){4}}
{\linethickness{0.6mm}  \put(9,25){\line(1,0){7}}}
\put(16,25){\line(1,-1){4}}
\put(5,14){\line(1,-1){4}}
{\linethickness{0.6mm} \put(9,10){\line(1,0){7}}}
\put(16,10){\line(1,1){4}}

\put(-3,17){$G~:$}
{\small \put(32,22){$\nu(G)+1$}
\put(45,22){ = 5}
\put(32,17){$\gdim(G)$}
\put(45,17){ = 5}
\put(27,12){$\gamma_{~\textup{P}}(G)/2+1$}
\put(45,12){ = 3}}
\end{picture}}
\qquad\qquad\qquad
{\setlength{\unitlength}{2mm}
\begin{picture}(25,15)
\put(5,5){\circle{0.75}}
\put(10,5){\circle{0.75}}
\put(15,5){\circle{0.75}}
\put(20,5){\circle{0.75}}
\put(5,15){\circle*{0.75}}
\put(10,15){\circle*{0.75}}
\put(15,15){\circle*{0.75}}
\put(20,15){\circle*{0.75}}

\put(5,15){\line(0,-1){9.625}}
\put(5,15){\line(1,-2){4.8}}
\put(10,15){\line(0,-1){9.625}}
\put(10,15){\line(1,-2){4.8}}
\put(15,15){\line(0,-1){9.625}}
\put(15,15){\line(1,-2){4.8}}
\put(20,15){\line(0,-1){9.625}}

\end{picture}}
\]
\[
{\setlength{\unitlength}{1.15mm}
\begin{picture}(25,25)

\put(5,14){\circle*{1}}
\put(5,21){\circle*{1}}
\put(9,10){\circle*{1}}
\put(9,25){\circle*{1}}
\put(16,10){\circle*{1}}
\put(16,25){\circle*{1}}
\put(20,14){\circle*{1}}
\put(20,21){\circle*{1}}

\put(5,14){\line(0,1){7}}
\put(5,21){\line(1,1){4}}
{\linethickness{0.6mm} \put(9,25){\line(1,0){7}}}
\put(16,25){\line(1,-1){4}}
\put(5,14){\line(1,-1){4}}
{\linethickness{0.6mm} \put(9,10){\line(1,0){7}}}
\put(16,10){\line(1,1){4}}
\put(20,14){\line(0,1){7}}

{\small \put(32,22){$\nu(G)+1$}
\put(45,22){ = 5}
\put(32,17){$\gdim(G)$}
\put(45,17){ = 4}
\put(27,12){$\gamma_{~\textup{P}}(G)/2+1$}
\put(45,12){ = 3}}
\end{picture}}
\qquad\qquad\qquad
{\setlength{\unitlength}{2mm}
\begin{picture}(25,15)
\put(5,5){\circle{0.75}}
\put(10,5){\circle{0.75}}
\put(15,5){\circle{0.75}}
\put(20,5){\circle{0.75}}
\put(5,15){\circle*{0.75}}
\put(10,15){\circle*{0.75}}
\put(15,15){\circle*{0.75}}
\put(20,15){\circle*{0.75}}

\put(5,15){\line(0,-1){9.625}}
\put(5,15){\line(1,-2){4.8}}
\put(10,15){\line(0,-1){9.625}}
\put(10,15){\line(1,-2){4.8}}
\put(15,15){\line(0,-1){9.625}}
\put(15,15){\line(1,-2){4.8}}
\put(20,15){\line(0,-1){9.625}}
\put(20,15){\line(-3,-2){14.625}}

\end{picture}}
\]
\[
{\setlength{\unitlength}{1.15mm}
\begin{picture}(25,25)

\put(5,14){\circle*{1}}
\put(5,21){\circle*{1}}
\put(9,10){\circle*{1}}
\put(9,25){\circle*{1}}
\put(16,10){\circle*{1}}
\put(16,25){\circle*{1}}
\put(20,14){\circle*{1}}
\put(20,21){\circle*{1}}

\put(5,14){\line(0,1){7}}
\put(5,21){\line(1,1){4}}
\put(5,21){\line(1,0){15}}
{\linethickness{0.6mm} \put(9,25){\line(1,0){7}}}
\put(9,25){\line(0,-1){15}}
\put(16,25){\line(1,-1){4}}
\put(16,25){\line(0,-1){15}}
\put(5,14){\line(1,-1){4}}
\put(5,14){\line(1,0){15}}
{\linethickness{0.6mm} \put(9,10){\line(1,0){7}}}
\put(16,10){\line(1,1){4}}
\put(20,14){\line(0,1){7}}

{\small \put(32,22){$\nu(G)+1$}
\put(45,22){ = 5}
\put(32,17){$\gdim(G)$}
\put(45,17){ = 3}
\put(27,12){$\gamma_{~\textup{P}}(G)/2+1$}
\put(45,12){ = 3}}
\end{picture}}
\qquad\qquad\qquad
{\setlength{\unitlength}{2mm}
\begin{picture}(25,15)
\put(5,5){\circle{0.75}}
\put(10,5){\circle{0.75}}
\put(15,5){\circle{0.75}}
\put(20,5){\circle{0.75}}
\put(5,15){\circle*{0.75}}
\put(10,15){\circle*{0.75}}
\put(15,15){\circle*{0.75}}
\put(20,15){\circle*{0.75}}

\put(5,15){\line(0,-1){9.625}}
\put(5,15){\line(1,-2){4.8}}
\put(5,15){\line(1,-1){9.7}}
\put(10,15){\line(0,-1){9.625}}
\put(10,15){\line(1,-2){4.8}}
\put(10,15){\line(1,-1){9.7}}
\put(15,15){\line(0,-1){9.625}}
\put(15,15){\line(1,-2){4.8}}
\put(15,15){\line(-1,-1){9.7}}
\put(20,15){\line(0,-1){9.625}}
\put(20,15){\line(-3,-2){14.625}}
\put(20,15){\line(-1,-1){9.7}}
\end{picture}}
\]
\end{example}
In spite of the examples above, the graphical dimension is easy to compute at least for trees. In this case $\gdim(G)=\nu(G)+1$ (Proposition \ref{tree}), and there are many algorithms that compute the matching number of a bipartite graph. 

To prove the main result of this section, we need the following lemma:
\begin{lemma}\label{pointcover}
Let $G$ be a graph, $k >0$ a natural number and $\aa$ a basic $k$-cover of $G$. Denote by 
$A_{k/2}:= \{ v\in V~:~ \aa(v) \le k/2\}$. The set 
$A_{k/2}$ is a point cover of $G$ and  $\aa$ is uniquely determined by the values it takes on the vertices in 
$A_{k/2}$.
\end{lemma}
\begin{proof}
Denote by $W :=  V \setminus A_{k/2} = \{ w \in V ~:~ \aa(w) >k/2\}$. As $\aa$ is basic, for each vertex $w \in W$ there exists a vertex $v$ such that $\{w,v\} \in E$ and $\aa(w) + \aa(v) = k$. As $\aa(w) > k/2$ we must have that $\aa(v) < k/2$. So    the set  $A_{k/2}$ is a point cover of $G$.

It is easy to see that the only possible choice to extend $\aa$ on the set $W$ is:
\[ \aa(w) = \max\{ k- \aa(v) ~:~ \{v,w\} \in E, \textup{~and~} v \in A_{k/2} \}.\]
As  $A_{k/2}$ is a point cover, the set we are considering is not empty for any $w \in W$. In order to obtain a $k$-cover, we need to assign to $\aa(w)$ at least the maximum considered above. But in order to obtain a basic $k$-cover we need to assign  exactly this value.
\end{proof}

Before stating the main theorem of this section we need the following crucial remark.

\begin{remark}\label{polynomial}
There exists a polynomial $P\in \mathbb{Q}[t]$ of degree $\dim (\AG)-1$ such that, for $h\gg 0$, 
\[P(h)=|\{\mbox{basic $2h$-covers of $G$}\}|.\] 
To see this, let us consider the second Veronese of $\AG$, namely $\AG^{(2)}=\oplus_{h\geq 0}\AG_{2h}$. By \cite[Theorem 5.1.a]{HHT} we have that $\AG^{(2)}$ is a standard graded $K$-algebra. So it has a Hilbert polynomial, denoted by $\HP$, such that $\HP(h)=\dim K(\AG_{2h})$ for $h\gg 0$. Notice that $\AG$ is a finite $\AG^{(2)}$-module, so $\dim(\AG)=\dim(\AG^{(2)})$, which is the degree the degree of $\HP$ minus $1$. So it is enough to take $P=\HP$.
\end{remark}

Now we can prove the main result of this section.
\begin{thm}\label{dim}
Let $\AG$ be the symbolic fiber cone of the cover  ideal of a graph $G$. Then
\[ \dim(\AG) = \gdim(G). \]
\end{thm}

\begin{proof} We will first prove that $\dim(\AG) \ge \gdim(G)$.
By Remark \ref{polynomial} we have to show that $|\{\mbox{basic $2h$-covers of $G$}\}|$ grows as a polynomial in $h$ of degree at  least $\gdim(G)-1$.

Let $A = \{a_1, \ldots, a_r\} \subseteq V$ be a free parameter set of maximal cardinality, $B = \{b_1, \ldots, b_r\}$ a partner set of $A$ and denote by $X = A \cup B$. So $\gdim(G) =r+1$.
Let $k > 2r$ be an even natural number. 
 We will construct a basic $k$-cover of $G$ for every decreasing sequence of numbers:
 \[ \frac{k}{2}\ge i_1 > i_2 > \ldots > i_r \ge 0.\] 
As the number of decreasing sequences as above is $\binom{k/2+1}{r}$, this will imply that the degree of $\HP$ is at least $r$, so also that $\dim(\AG) \ge \gdim(G)$. 
 For a decreasing sequence as above and  for all $j \in [r]$ we define:
\begin{eqnarray*}
 \aa(a_j)& := & i_j,\\
 \aa(b_j)&:=& k-i_j.
 \end{eqnarray*} 
As $G$ is connected, if $V \setminus X \neq \emptyset$, there exists a vertex $v \in  V \setminus X$ such that there exists at least one edge between $v$ and $X$. We define:
\[ \aa(v): = \max\{ k - \aa(w) ~:~ w \in X \textup{~and~} \{v,w\} \in E\},\]
append $v$ to $X$ and continue in the same way until $\aa$ is defined for all vertices of $G$.
It is easy to see that by construction, for each edge $\{v,w\}$ with $v \notin X$ or $w \notin X$ (or both), we have $\aa(v) + \aa(w) \ge k$ and that for each vertex $v \notin X$ there exists another vertex $v'$ such that $\aa(v) + \aa(v') = k$. So to check that we defined a basic $k$-cover we need to focus on the  vertices in $X$.

Let $\{v,w\}$ be an edge with $v,w \in X$. As $A$ is a set of independent vertices, we can assume that $w = b_j \in B$ and check the following two cases:\\
If $v = a_h \in A$ then by definition $h \le j$, and by construction:
\[ \aa(a_h) + \aa(b_j) = i_h + k - i_j \ge k.\]
If $v = b_l \in B$ then:
\[ \aa(b_l) + \aa(b_j) = k - i_l + k - i_j \ge k.\]
So $\aa$ is a $k$-cover. The fact that $\{a_j,b_j\} \in E$ for each $1\le j \le r$  guarantees that $\aa$ is a basic $k$-cover.

Assume now that $\dim(\AG) = s+1$. To prove that $\dim(\AG) \le \gdim(G)$ we will use the following:

\emph{Claim:} For some  natural number $k \gg0$ there exists a basic $k$-cover $\aa$ such that there are at least $s$ different values of $\aa$ which are smaller than $k/2$, namely:
\[ |\{ \aa(v) ~:~ v \in V \textup{~and~} \aa(v)  \le k/2\}| \ge s.\]
Suppose the claim is true and let $\aa$ be a basic $k$-cover as above. Denote by 
\[ \{i_1, \ldots, i_r\} := \{ \aa(v) ~:~ v \in V \textup{~and~} \aa(v)  \le k/2\}.\]
By the claim $r\ge s$. We can also assume that $i_1 > i_2 > \ldots > i_r$. For each $1\le j \le r$  choose a vertex $a_j \in V$ such that $\aa(a_j)= i_j$
and denote by
 \[A := \{a_1, \ldots, a_r\}.\] 
 As $\aa$ is a basic $k$-cover, for each $1 \le j \le r$ there exists a vertex $b_j \in V$ such that $\aa(a_j) + \aa(b_j) = k$. Choose one such $b_j$ for each $j$ and denote by 
\[ B := \{b_1, \ldots, b_r\}.\]  
 It is not difficult to see that $A$ is a free parameter set with the partner set  $B$, so 
 \[ \gdim(G) \ge r+1 \ge s+1 = \dim(\AG).\]
So we only need to prove the claim.

Suppose there is no  $\aa$ as we claim. Then, for every $k\geq 0$, there is an injection 
\[ \{\mbox{basic $k$-covers of $G$}\} \hookrightarrow \{(a_1, \ldots , a_n): \ 0 \leq a_i \leq k/2 \mbox{ and }|\{a_1, \ldots , a_n\}|<s\},\]
Using the notation of Lemma \ref{pointcover}, the application above is given by associating to each basic $k$-cover $\aa$, a vector which has the same values as $\aa$ on   $A_{k/2}$ and is 0 in all the other positions. Lemma \ref{pointcover} guarantees that this is an injection.

It is not difficult to see that the cardinality of the set on the right-hand side is equal to $C\cdot(k+1)^{s-1}$, where $C$ is a constant depending on $n$ and $s$. Therefore Remark \ref{polynomial} implies $\dim \AG \leq s$, a contradiction.

\end{proof}

We recall that the analytic spread of a homogeneous ideal $I\subseteq S$, denoted by $\ell(I)$, is the dimension of its ordinary fiber cone. When $K$ is an infinite field,  Northcott and Rees proved in \cite{NR} that  $\ell(I)$ is the cardinality of a set of minimal generators of a minimal reduction of $IS_{\mm}$, i.e. an ideal $\mathfrak{a}\subseteq S_{\mm}$ minimal by inclusion and such that there exists $k$ for which $\mathfrak{a}(IS_{\mm})^k=(IS_{\mm})^{k+1}$.

\begin{corollary}
Let $G$ be a bipartite graph. Then
\[\ell(J(G))=\gdim(G).\]
\end{corollary}
\begin{proof}
As said in the preliminaries, in \cite[Theorem 5.1.b]{HHT} the authors showed that $G$ is bipartite if and only if $A(G)$ is a standard graded $S$-algebra. This is equivalent to $A(G)$ being the ordinary Rees algebra of $J(G)$. Therefore, when $G$ is bipartite, $\AG$ is the ordinary fiber cone of $J(G)$, so the corollary follows by Theorem \ref{dim}.
\end{proof}

Before we state the next proposition, let us establish some notation that we will use in its proof. Let $G$ be a bipartite graph  with bipartition of the vertex set $V_1 \cup V_2$. In order to compute the graphical dimension we only need to look at free parameter sets $A_0 \subseteq V_1$ with partner sets $B_0\subseteq V_2$.
Notice that the graph induced by the  set of vertices  $A_0 \cup B_0$ may not be connected. Denote this graph by $G_0$ and  denote its connected components  by $C_1, C_2,$ and so on. Notice that if $G$ is a tree, then for any vertex $v \notin A_0 \cup B_0$, if there exists  an edge $\{v,w_0\}$, with $w_0 $ in some $ C_i$, then $\{v,w\}$ is not an edge for any $w\in C_i$, $w\neq w_0$. In other words, a vertex outside $G_0$ is ``tied" to a connected component of $G_0$ by at most one edge.  
\begin{prop}\label{tree}
If $ G$ is  a tree, then
$ \dim \AG =\nu(G)+1$, where $\nu(G)$ is the matching number of $G$.
\end{prop}
\begin{proof}
By Remark \ref{bounds} and Theorem \ref{dim} we only have to prove that $\gdim(G) \geq \nu(G)+1$ whenever $G$ is a tree. 
Choose $A_0=\{a_1,\ldots, a_r\}$ a maximal free parameter set with partner set $B_0=\{b_1,\ldots,b_r\}$
 and suppose that the matching $M = \{ \{a_ib_i\} \}_{i=1,\ldots,r}$ is not maximal. \\
 By a classical result of Berge (for instance see the book of Lov\'asz and Plummer \cite[Theorem 1.2.1]{LP})  we get that there must exist an augmenting path in $G$ relative to $M$. As $G$ is bipartite  it is easy to see that this path must be of the form $P = a', b_{i_1}, a_{i_1},\ldots, b_{i_k},a_{i_k},b',$ and  as $A_0$ is a free parameter set the indices must be ordered in the following way $1\le i_1<\ldots< i_k \le r.$  We will construct a new free parameter set with $r+1$ elements. Notice that $a'$ and $b'$ are not vertices of $G_0$.\\
  Denote by $C$ the connected component of $G_0$  to which the vertices in $P\cap(A_0\cup B_0) $ belong.
We reorder the connected components such that the $C_i$'s to which $b'$ is connected come first, $C$ comes next and the connected components to which $a'$ is connected come last. Inside $C$ we relabel the vertices such that $a_{i_k}, a_{i_{k-1}}, \ldots, a_{i_1}, a'$ are the first $k+1$ with partners $b', b_{i_k}, \ldots, b_{i_2}, b_{i_1}$. It is easy to see now that, as there are no cycles in $G$,  we obtain a new free parameter set of cardinality $r+1$, a contradiction. 
\end{proof}

Given an ideal $I$ of some ring $R$ we recall that the arithmetical rank of $I$ is the integer
\[\ara(I)=\min \{r: \ \exists \ f_1, \ldots , f_r \in R \mbox{ for which }\sqrt{I}= (f_1, \ldots , f_r)\}.\]
If $R$ is a factorial domain, geometrically $\ara(I)$ is the minimal number of hypersurfaces that define set-theoretically the scheme $\mathcal{V}(I)$ in $\operatorname{Spec}(R)$. As we said in the beginning of this section we can obtain interesting upper bounds for this number in the case of monomial ideals of pure codimension $2$ in $S_{\mm}$.

\begin{corollary}\label{ara}
Let $K$ be an infinite field, and $G$ a graph. Then
\[\ara(J(G)S_{\mm})\leq \gdim(G).\]
In particular, $ \ara(J(G)S_{\mm})\leq \nu(G)+1$. 
\end{corollary}
\begin{proof}
Let us consider the second Veronese of $\AG$, i.e. 
\[ \AG^{(2)}=\bigoplus_{i\geq 0}\AG_{2i}. \]
By \cite[Theorem 5.1.a]{HHT} we have $J(G)^{(2i)}=(J(G)^{(2)})^i$, so that $\AG^{(2)}$ is the ordinary fiber cone of $J(G)^{(2)}$. Since $\AG$ is finite as a $\AG^{(2)}$-module, the Krull dimensions of $\AG$ and the one of $\AG^{(2)}$ are the same. Therefore, using Theorem \ref{dim}, we get
\[ \gdim(G)=\dim \AG^{(2)}=\ell(J(G)^{(2)})=\ell((J(G)S_{\mm})^{(2)}). \]
By a result in \cite[p.151]{NR}, since $K$ is infinite, the analytic spread of $(J(G)S_{\mm})^{(2)}$ is the cardinality of a set of minimal generators of a minimal reduction of it. The radical of such a reduction is clearly the radical of $(J(G)S_{\mm})^{(2)}$, i.e. $J(G)S_{\mm}$. So we get the desired inequality.
\end{proof}

\begin{remark}
The author of \cite{Ly1} proved that the arithmetical rank of a monomial ideal of pure codimension $2$, once localized at $\mm$, is at most $\lfloor n/2 \rfloor + 1$, where $n$ is the numbers of variables. But every squarefree monomial ideal of codimension $2$ is obviously of the form $J(G)$ for some graph on $[n]$. So, since $\nu(G)$ is at most $\lfloor n/2 \rfloor$, Corollary \ref{ara} refines the result of Lyubeznik.
\end{remark}

\begin{corollary}
Let $G$ be a graph for which $\gdim(G)-1$ is equal to the maximum size of a set of pairwise disconnected edges, then
\[\ara(J(G)S_{\mm}) = \gdim(G).\]
\end{corollary}
\begin{proof}
By a result of Katzman (\cite[Proposition 2.5]{Ka}) the maximum size of a set of pairwise disconnected edges of $G$ provides a lower bound for the Castelnuovo-Mumford regularity of $S/I(G)$. Therefore, $\reg(S/I(G))\geq \gdim(G)-1$. But $J(G)$ is the Alexander dual of $I(G)$, so a result of Terai (\cite{Te}) implies that $\pd(S/J(G))\geq \gdim(G)$. Now, Lyubeznik showed in \cite{Ly} that $\pd(S/I)=\operatorname{cd}(S,I)=\operatorname{cd}(S_{\mm},IS_{\mm})$ (cohomological dimension) for any square-free monomial ideal $I$. Since the cohomological dimension provides a lower bound for the arithmetical rank, we get $\ara(J(G)S_{\mm})\geq \gdim(G)$. Now we get the conclusion by Corollary \ref{ara}.
\end{proof}

\begin{corollary}
Let $I\subseteq S=K[x_1,\ldots ,x_n]$ be  a square-free monomial ideal of pure codimension $2$, and let $d$ be the minimum degree of a non zero monomial in $I$. Assume that the field $K$ is infinite. Then 
\[\ara(IS_{\mm})\leq \min \{d+1,n-d+1\}\]
\end{corollary}
\begin{proof}
The inequality $\ara(IS_{\mm})\leq n-d+1$ is well known. One way to see this is by defining the following partial order on the set of the square-free monomials of $S$:
\[ m\leq n \ \ \iff \ \ n|m \ \ \mbox{ \ \ for any square-free monomials $m,n$ of }S. \]
It is easy to see that $S$ is an algebra with straightening laws (not homogeneous -- see \cite{BrVe} for the definition) on this poset over $K$. Notice that $I$ comes from a poset ideal. This means that $I=\Omega S$, where $\Omega$ is a subset of the square-free monomials such that: $n \in \Omega$, $m\leq n$ $\implies$ $m\in \Omega$. Then by \cite[Proposition 5.20]{BrVe} we get $\ara(I)\leq n-d+1$. This obviously implies that $\ara(IS_{\mm})\leq n-d+1$.

To prove the inequality $\ara(IS_{\mm})\leq d+1$, notice that $I=J(G)$  for a graph $G$ on $[n]$ ($\{i,j\}$ is an edge of $G$ if and only if $(x_i,x_j)$ is a minimal prime of $I$). Then Corollary \ref{ara} implies that $\ara(IS_{\mm})\leq \nu(G)+1$. It is well known and easy to show, that the matching number is at most the least cardinality of a vertex cover of $G$. It turns out that this number 
is equal to $d$.
\end{proof}

\section{Cohen-Macaulay Property and Castelnuovo-Mumford Regularity of the Edge Ideal}

An interesting  open problem, far to be solved, is to characterize in a combinatorial fashion all the Cohen-Macaulay graphs. The authors of \cite{HH} gave a complete answer when $G$ is bipartite. On the other hand if $G$ is Cohen-Macaulay then it is unmixed, and for bipartite unmixed graphs $\AG$ is the ordinary fiber cone of an ideal generated in one degree, so it   is a domain. This means that a bipartite Cohen-Macaulay graph satisfies the WSC. Since many of these graphs are not bipartite (see \cite{BV} for details), a natural extension of the theorem of Herzog and Hibi would be characterize all the  graphs satisfying the WSC which are Cohen-Macaulay. We are able to do this defining for each graph $G$ a \lq\lq nicer\rq\rq  ~graph $\pi(G)$. This association behaves like a projection. 

We start with a definition that makes sense by \cite[Lemma 2.1]{BV}.
\begin{definition}
We say that an edge $\{i,j\}$ of $G$ is a {\it right edge} if one of the following equivalent conditions is satisfied:
\begin{compactenum} 
\item for any basic $1$-cover $\aa$ of $G$ we have $\aa(i)+\aa(j)=1$;
\item for any basic $k$-cover $\aa$ of $G$ we have $\aa(i)+\aa(j)=k$;
\item if $\{i,i'\}$ and $\{j,j'\}$ are edges of $G$, then $\{i',j'\}$ is an edge of $G$ as well (in particular $i' \neq j'$).
\end{compactenum}
\end{definition}
Notice that a graph satisfies the WSC if and only if every vertex belongs to a right edge. We recall that these graphs are  of interest  because they are exactly  those graphs for which $\AG$ is a domain. In \cite{BV} the authors constructed from $G$ a graph $G^{0-1}$, possibly with isolated vertices, in order to characterize the graphs for which all the symbolic powers of $J(G)$ are generated in one degree.
We recall the definition:
\begin{compactenum} 
\item $V(G^{0-1})=V(G)$;
\item $E(G^{0-1})=\{ \{i,j\} \in E(G): \ \{i,j\} \mbox{ is a right edge of $G$}  \}$.
\end{compactenum}
It was proved in \cite{BV} that for any $G$ the graph $G^{0-1}$ is the disjoint union of some complete bipartite graphs $K_{a,b}$ (with $b\geq a\geq 1$) and some isolated points. Moreover $G^{0-1}$ has no isolated vertices if and only if $G$ satisfies the WSC.

We construct a new graph, that we will denote by $\pi(G)$, as follows: assume that
\[ G^{0-1}=\Big(\bigcup_{i=1}^m K_{a_i,b_i}\Big)~\bigcup ~\Big(\bigcup_{i=1}^t \{v_i\}\Big), \]
where the unions are disjoint unions of graphs, $b_i\geq a_i \geq 1$ and $v_i \in V(G)$. Denoting by $(A_i,B_i)$ the bipartition of $K_{a_i,b_i}$, we define the vertex set of $\pi(G)$ as 
\[ V(\pi(G))=\{A_i,B_i,\{v_j\}: \ i=1, \ldots ,m \mbox{ and }j=1,\ldots , t\}.\]
The edges of $\pi(G)$ are defined as follows: if $U,W$ belong to $V(\pi(G))$, then $\{U,W\}\in E(\pi(G))$ if and only if there is an edge of $G$ connecting a vertex of $U$ with a vertex of $W$.  By \cite[Lemma 2.6]{BV} the existence of an edge from $U$ to $W$ is equivalent to the fact that the induced subgraph of $G$ on the vertices of $U \cup W$ is bipartite complete. By \cite[Lemma 2.6.(1)]{BV} $\pi(G)$ has no loops. The notation $\pi$ comes from the fact that the operator $\pi$ is a projection, in the sense that $\pi(\pi(G))=\pi(G)$.
 
The following result is one of the reasons for introducing $\pi(G)$.
\begin{prop}\label{isom}
For every graph $G$, there is a well defined 1-1 correspondence 
\[\pi: \{\mbox{basic covers of }G\}\longrightarrow \{\mbox{basic covers of }\pi(G)\}\]
that associates to a basic $k$-cover $\aa$ of $G$ the basic $k$-cover $\pi(\aa)$ of $\pi(G)$, with $\pi(\aa)(U)= \aa(u)$ for some $u\in U$.
Moreover this correspondence induces a graded isomorphism
\[ \AG \cong \bar{A}(\pi(G)).\]
\end{prop}
\begin{proof}
Using the fact that the edges between each $A_i$ and $B_i$ are right, it is straightforward to check that $\aa$ has the same value on all vertices in $A_i$ (resp. in $B_i$) for every $i=1, \ldots ,m$.
This implies  that $\pi$ is a well defined function. It is easy to see that $\pi$ is a bijection between the basic $k$-covers of $G$ and those of $\pi(G)$; moreover this operation is compatible with the multiplicative structure on $\AG$ and of $\bar{A}(\pi(G))$. Therefore we also have a graded isomorphism between the algebras $\AG$ and $\bar{A}(\pi(G))$.
\end{proof}

\begin{remark}
\begin{itemize}
\item[1.]
The previous Proposition provides another proof of the fact that $\AG$ is a Hibi ring when $G$ is a bipartite graph satisfying the WSC. In fact in this case $\pi(G)$ is unmixed bipartite, so it is known that $\bar{A}(\pi(G))$ is a Hibi ring (for instance see \cite[Theorem 3.3]{BCV}). 
\item[2.]
Proposition \ref{isom} shows also that $\P(G)=\P(\pi(G))$. So in order to study $\P(G)$ it can be convenient to pass to the projection and work on  a graph with less vertices.
\end{itemize}
\end{remark}

In some cases $\pi(G) = G$, for instance if $G$ is a cycle on $n\neq 4$ vertices. The usefulness of $\pi(G)$ arises especially when $G$ satisfies the WSC. As we already said in the above remark,  in this case $\pi(G)$ is unmixed. Less trivially, we can strengthen this fact, but first we need a technical lemma.

\begin{lemma}\label{poset}
Let $G$ be a graph satisfying the WSC. Then there exists a unique perfect matching $M=\{\{u_i,v_i\}:i=1, \ldots, r\}$ of $\pi(G)$, where $r=|\pi(G)|/2$. Moreover it is possible to label the vertices of $\pi(G)$ in such a way that $\{v_1, \ldots ,v_r\}$ is an independent set of vertices of $\pi(G)$ and that the relation $v_i\prec v_j$ if and only $\{u_i,v_j\}$ is an edge defines a partial order on $V=\{v_1, \ldots ,v_r\}$.
\end{lemma}
\begin{proof}
Since $G$ satisfies the WSC,  $G^{0-1}$ has no isolated points, so we obtain a perfect matching $M=\{\{u_i,v_i\}:i=1, \ldots ,r\}$ directly by  construction. Moreover, since the edges of $M$ are right, it immediately follows that for each $1$-cover $\aa$ of $\pi(G)$ we have $\sum_{v\in \pi(G)}\aa(v)=r$. This implies that if $N$ is another perfect matching of $\pi(G)$ then the $r$ edges of $N$ must be right. But the only right edges of $\pi(G)$ are those of $M$, therefore $M=N$.  

We  prove now that we can assume that $\{v_1, \ldots ,v_r\}$ is an independent set of vertices. In fact, suppose that there exist $i<j$ such that $\{v_i,v_j\}$ is an edge, and take the least $j$ with this property. 
First notice that there exists no edge $\{u_j,v_k\}$ of $\pi(G)$ with $k<j$.  The existence of such an edge  would imply that also $\{v_k,v_i\}$ is an edge (as $\{u_j,v_j\}$ is right) and this would contradict the minimality of $j$.
Now  switch $v_j$ and $u_j$. As we have seen that  there are no edges $\{u_j,v_k\}$ with $k<j$,
  we can proceed with the same argument and assume that $\{v_1, \ldots ,v_r\}$ is an independent set of vertices.

To conclude we have to  show that the relation
\[v_i\prec v_j \ \iff \ \{u_i, v_j\} \mbox{ is an edge of }\pi(G)\]
defines a partial order on $V$.
\begin{compactenum}
\item {\it Reflexivity} is obvious.
\item {\it Transitivity} is straightforward because $\{u_i,v_i\}$ is a right edge of $\pi(G)$, $\forall$ $i=1, \ldots ,r$. 
\item {\it Anti-symmetry}: suppose there exist $i\neq j$ such that $v_i\prec v_j$ and $v_j \prec v_i$. Then $\{u_i,v_j\}$ and $\{u_j,v_i\}$ are both edges of $\pi(G)$. This contradicts \cite[Lemma 2.6, point (3)]{BV}.
\end{compactenum} 
\end{proof}

We recall that if $I\subseteq S$ is a square-free monomial ideal we can associate to it the simplicial complex $\Delta(I)$ on the set $[n]$ such that $\{i_1, \ldots ,i_s\}$ belongs to $\Delta(I)$ if and only if $x_{i_1} \cdots x_{i_s}$ does not belong to $I$.

To prove the next result we need a theorem  from \cite{MRV}, that we are going to state in the case of graphs. We recall that a graph $G$ has a perfect matching of \emph{K\"onig type} if it has a perfect matching of cardinality $\height(I(G))$.

\begin{thm}\label{MRV} (Morey, Reyes and Villareal \cite[Theorem 2.8]{MRV}). Let $G$ be an unmixed graph which admits a matching of K\"onig type. Assume that for any vertex $v$ the induced subgraph on all the vertices of $G$ but $v$ has a leaf. Then $\Delta(I(G))$ is shellable. 
\end{thm}

Thus we are ready to show the following.

\begin{thm}\label{CM}
Let $G$ be a graph satisfying the WSC, and let $\Delta=\Delta(I(\pi(G))$. Then $\Delta$ is shellable. In particular $\pi(G)$ is a Cohen-Macaulay graph.
\end{thm}
\begin{proof}
We want to use Theorem \ref{MRV}. It is clear  that $\pi(G)$ is unmixed because it has a perfect matching of right edges. Furthermore such a matching is obviously of K\"onig type.  It remains to show that for any $v \in V(\pi(G))$, the induced subgraph of $\pi(G)$ on $V(\pi(G))\setminus \{v\}$ has a leaf. Label the vertices of $\pi(G)$ as in Lemma \ref{poset} and in such a  way that $v_i \prec v_j$ provided that $i\leq j$. Since $v_1$ is a leaf, the only problem could arise when we remove from $\pi(G)$ either  $u_1$ or $v_1$. If we remove $u_1$, then  $v_2$ becomes a leaf, so we must show that the graph induced by $\pi(G)$ on $V(\pi(G))\setminus \{v_1\}$ has a leaf. \\
Suppose there are no leaves. Then, denoting by $r=|V(\pi(G))|/2$, we can choose the minimum $i$ such that $\{u_i,u_r\}$ is an edge (because $u_r$ is not a leaf and by Lemma \ref{poset} these are the only possible edges, different from $\{u_r,v_r\}$, containing $u_r$). We claim that $i=1$. If not, since $v_i$ is not a leaf, there exists $j<i$ such that $\{u_j,v_i\}$ is an edge. But, since $\{u_i,v_i\}$ is a right edge, it follows that $\{u_j,u_r\}$ is an edge, contradicting the minimality of $i$. Now, since $v_r$ is not a leaf, there exists a minimal $k<r$ such that $\{u_k,v_r\}$ is a leaf. Arguing as above we have that $k=1$. Then $\{u_1,u_r\}$ and $\{u_1,v_r\}$ are both edges, and this contradicts the fact that $\{u_r,v_r\}$ is right.

Therefore $\Delta$ is shellable by Theorem \ref{MRV}, and it is well known that this implies that $\pi(G)$ is a Cohen-Macaulay graph (for instance see the book of Bruns and Herzog \cite[Theorem 5.1.13]{BH}).  
\end{proof}

For the following result we recall that an ideal $I\subseteq S$ is connected in codimension $1$ if any two minimal primes $\wp,\wp'$ of $I$ are $1$-connected: i.e. there exists a path $\wp=\wp_1,\ldots ,\wp_m=\wp'$ of minimal primes of $I$ such that $\height(\wp_i+\wp_{i+1})=\height(I)+1$. If $I=I_{\Delta}$ is a square-free monomial ideal, then $I$ is connected in codimension $1$ if and only if $\Delta$ is strongly connected, i.e. if and only if you can walk from a facet to another passing through faces of codimension $1$ in $\Delta$.

\begin{thm}
Let $G$ be a graph satisfying the WSC and set $\Delta=\Delta(I(G))$ the simplicial complex associated to the edge ideal. The following  conditions are equivalent:
\begin{compactenum}
\item $G$ has a unique perfect matching;
\item $G$ has a unique perfect matching of right edges;
\item $\pi(G)=G$;
\item $\Delta$ is shellable;
\item $G$ is Cohen-Macaulay;
\item $I(G)$ is connected in codimension 1.
\end{compactenum}
\end{thm}
\begin{proof}
\mbox{$3.\Rightarrow4.$}  is Theorem \ref{CM}. \mbox{$4.\Rightarrow5.$}  follows by  \cite[Theorem 5.1.13]{BH}. \mbox{$5.\Rightarrow6.$}  is a general fact proved by Hartshorne in \cite{Ha1}. \mbox{$3.\Rightarrow2.$}  follows  immediately from Lemma \ref{poset}.

We want to show that \mbox{$6.\Rightarrow3.$}  Suppose $\pi(G)\neq G$. This means that there is a bipartite complete subgraph of $G$, say $H$, with more than two vertices and such that any edge of $H$ is a right edge of $G$. Let $V(H)=A\cup B$ be the bipartition of the vertex set of $H$, and assume that $|A|\geq 2$. It is easy to construct a basic $1$-cover $\aa$ that associates  $1$ to the vertices in $A$ and $0$ to the vertices in $B$, and a basic $1$-cover $\bb$ that associates $0$ to the vertices in $A$ and $1$ to the ones in $B$. Consider the two ideals of $S$
$$
\begin{array}{ccc}
\wp_{\aa}&=&(x_i: \aa(i)=1)\\
\wp_{\bb}&=&(x_i: \bb(i)=1).
\end{array}
$$
The ideals $\wp_{\aa}$ and $\wp_{\bb}$ are minimal prime ideals of $I(G)$. We claim that they are not $1$-connected. If they were, there would be a minimal prime ideal $\wp$ of $I(G)$ such that there exist $i,j\in A$ with $x_i\in \wp$ and $x_j\notin \wp$. Therefore the basic $1$-cover $\cc$ associated to $\wp$ with $\cc(i) =1$ and $\cc(j)=0$. Because $\cc$ is a $1$-cover, it  must also associate $1$ to every vertex of $B$, and this contradicts the fact that $H$ consists of right edges.

Now we are going to show that \mbox{$2.\Rightarrow3.$}  If $G$ has a perfect matching of right edges it is straightforward to check that it is unmixed. By \cite[Theorem 2.8]{BV},  the connected components of $G^{0-1}$ are all of the type $K_{a,a}$ for some $a\geq 1$. If $G$ were different from $\pi(G)$, then at least one of the $a$'s would be greater than 1. So we could find another perfect matching of right edges of $G$ by changing the matching of $K_{a,a}$ induced by the initial matching on $G$. 

For the implication \mbox{$1.\Rightarrow2$}, let $M=\{\{a_1,b_1\},\ldots ,\{a_m,b_m\}\}$ be the unique perfect matching of $G$. Suppose that an edge in $M$, say $\{a_1,b_1\}$, is not right. Since $G$ satisfies the WSC there is an $i>1$ such that $\{a_1,b_i\}$ (resp. $\{a_1,a_i\}$) is a right edge. But then $\{b_1,a_i\}$ (resp. $\{b_1,b_i\}$) is  an edge by the weak square condition. So $M'=\{\{a_1,b_i\},$ $\{a_2,b_2\},\ldots ,$ $\{a_i,b_1\},$ $ \ldots ,$ $\{a_m,b_m\}\}$ (resp. $M'=\{\{a_1,a_i\},$ $\{a_2,b_2\},\ldots ,\{b_i,b_1\},$ $ \ldots ,\{a_m,b_m\}\}$) is another matching, a contradiction.

It remains to show that \mbox{$2.\Rightarrow1.$}  But we already proved that if  2. holds then $G$ is Cohen-Macaulay. In particular $G$ is unmixed, so any other perfect matching of $G$ is forced to consist of right edges.
\end{proof}

Whereas graphs  whose edge ideal has a linear resolution have been completely characterized by Fr\"oberg in \cite{Fr}, it is still an open problem (even in the bipartite case) to characterize in a combinatorial fashion the Castelnuovo-Mumford regularity of the edge ideal. A general result in \cite{Ka} asserts that a lower bound for $\reg(S/I(G))$ is the maximum size of a pairwise disconnected set of edges of $G$. Moreover by the present paper it easily follows that the graphical dimension of $G$ provides an upper bound for $\reg(S/I(G))$ (see the remark below). In \cite{Zh} Zheng showed that if $G$ is a tree, then $\reg(S/I(G))$ is actually equal to the maximum number of disconnected edges of $G$. Later, in \cite{HV}, H\`a and Van Tuyl showed that the same conclusion holds true for chordal graphs, and recently, the author of \cite{Ku} showed this equality in the bipartite unmixed case, too. 
As another application of the operator $\pi$, we show in Theorem \ref{cmr} that  this equality holds also for any bipartite graph satisfying the WSC, extending the result of Kummini. First notice that to prove his theorem Kummini defines a new graph, called the acyclic reduction, starting from a bipartite unmixed graph (\cite[Discussion 2.8]{Ku}). It is possible to show that this new graph coincides with $\pi(G)$. So in some sense $\pi(G)$ can be seen as an extension to the class of all graphs of the acyclic reduction defined in \cite{Ku}.

\begin{remark}
We showed in Corollary \ref{ara} that, for any graph $G$, we have $\ara(J(G))\leq \gdim(G)$. But by a result in \cite{Ly} the cohomological dimension of $J(G)$ is equal to the projective dimension of $S/J(G)$. Since the cohomological dimension is a lower bound for the arithmetical rank, we have that $\pd(S/J(G))\leq \gdim(G)$. As $I(G)$ is the Alexander dual of $J(G)$, it follows by \cite{Te} that
\[\reg(S/I(G))\leq \gdim(G)-1.\]
Since $\gdim(G)-1$ is less than or equal to the matching number of $G$ by definition, the above inequality strengthens \cite[Theorem 1.5]{HV}.  
\end{remark}

\begin{lemma}\label{reg}
Let $G$ be any graph. Then
\[ \reg(S/I(G))=\reg(S'/I(\pi(G)), \]
where $S'=K[y_1,\ldots ,y_p]$ is the polynomial ring in $p=|V(\pi(G))|$ variables over $K$. 
\end{lemma}
\begin{proof}
For any $i=1,\ldots ,p$ call $V_i$ the set of vertices of $G$ that collapses to the vertex $i$ of $\pi(G)$. Then consider the homomorphism
$$
\begin{array}{rcl}
\phi: S' & \longrightarrow & S\\
 y_i & \mapsto & \prod_{j\in V_i}x_j=:m_i
\end{array}
$$
By the correspondence of basic $1$-covers of $G$ and $\pi(G)$ described in Proposition  \ref{isom}, one easily sees that $\phi(J(\pi(G)))S=J(G)$. Moreover it is obvious that $m_1,\ldots ,m_p$ form a regular sequence of $S$, so by a theorem of Hartshorne (\cite[Proposition 1]{Ha2}) $S$ is a flat $S'$-module via  $\phi$. Then if $F_{\bullet}$ is a minimal free resolution of $S'/J(\pi(G))$ over $S'$ it follows that $F_{\bullet}\otimes_{S'}S$ is a minimal free resolution of $S/J(G)$ over $S$. Therefore the total Betti numbers of $S'/J(\pi(G))$ and of $S/J(G)$ are the same, and in particular $\pd(S/J(G))=\pd(S'/J(\pi(G)))$. Thus \cite{Te} yields the conclusion.
\end{proof}

\begin{thm}\label{cmr}
Let $G$ be a bipartite graph satisfying the WSC. Then the Castelnuovo-Mumford regularity of $S/I(G)$  is equal to the maximum size of a pairwise disconnected set of edges of $G$.
\end{thm}
\begin{proof}
By Lemma \ref{reg}, using the same  notation, $\reg(S/I(G))=\reg(S'/I(\pi(G)))$. Moreover,  the maximum size of a pairwise disconnected set of edges in $G$ is equal to the same number for $\pi(G)$. Since $\pi(G)$ is Cohen-Macaulay by Theorem \ref{CM}, one can deduce the conclusion using \cite[Corollary 2.2.b]{HH}.
\end{proof}

\end{document}